\documentclass[10pt,a4paper]{article}

\usepackage[letterpaper,left=1in,right=1in,top=0.8in,bottom=0.8in]{geometry}
\usepackage{amssymb}
\usepackage[all]{xy}
\usepackage[english]{babel}
\usepackage{graphicx} 
\usepackage{enumerate}
\usepackage{amsthm}
\usepackage{amsmath}
\usepackage{mathtools}
\usepackage{tikz-cd}
\usepackage{bbm}
\usepackage{enumitem}

\usepackage[colorlinks=false,plainpages,backref,urlcolor=blue, hidelinks]{hyperref}
\usepackage[normalem]{ulem}
\usepackage{comment}

\theoremstyle{plain}
\newtheorem{Th}{Theorem}[section]
\newtheorem{Lemma}[Th]{Lemma}
\newtheorem{Cor}[Th]{Corollary}
\newtheorem{Prop}[Th]{Proposition}

\theoremstyle{definition}
\newtheorem{Def}[Th]{Definition}
\newtheorem{Conj}[Th]{Conjecture}

\newtheorem{Rem}[Th]{Remark}
\newtheorem{?}[Th]{Question}
\newtheorem{Ex}[Th]{Example}

\newtheorem{Axiom}[Th]{Axiom}

\newcommand{\NOTE}[1]{{\color{red}(#1)}}

\def\<{\langle}
\def\>{\rangle}

\newcommand{\Aa}{{\mathbb A}}
\newcommand{\Var}{\operatorname{Var}}
\newcommand{\Spec}{\operatorname{Spec}}
\newcommand{\Con}{\operatorname{Con}}
\newcommand{\Maps}{\operatorname{Maps}}

\newcommand{\End}{\operatorname{End}}
\newcommand{\GW}{\operatorname{GW}}
\newcommand{\SH}{\operatorname{SH}}
\newcommand{\one}{\mathbbm{1}}

\newcommand{\chic}{\operatorname{\chi_c}}
\renewcommand{\P}{{\mathbb{P}}}

\newcommand{\Sch}{\operatorname{\mathbf{Sch}}}
\newcommand{\Sm}{\operatorname{\mathbf{Sm}}}
\newcommand{\sU}{\mathcal{U}}
\newcommand{\sC}{\mathcal{C}}

\newcommand{\sA}{\mathcal{A}}

\numberwithin{equation}{section}

\begin{document}

\nocite{*} 

\title{\textsc{Motivic characteristic classes for singular schemes}}

\author{\textsc{\normalsize Ran Azouri}}

\date{} 

\maketitle

\begin{abstract}
We construct characteristic classes for singular algebraic varieties in motivic Borel-Moore homology, extending the motivic Euler class of the tangent bundle defined for smooth varieties. The two classes we define refine, in the setting of motivic homotopy theory, the top degree of the class constructed by Brasselet-Sch\"urmann-Yokura in $G$-theory, and the top degree of the pro-CSM class constructed by Aluffi in pro-Chow groups. We deduce an extension of the motivic Gauss-Bonnet formula to non-smooth proper varieties.
\end{abstract}

\setcounter{tocdepth}{1}
\tableofcontents

\section{Introduction}

One cornerstone for defining characteristic classes on singular spaces is MacPherson's work \cite{Mac}, affirming, in characteristic $0$, a conjecture of Deligne and Grothendieck on the existence of Chern classes for singular algebraic varieties. It involves the construction of a natural transformation between the functors of constructible functions and singular homology on $X$, a complex compact algebraic variety
\[ \Con(X) \rightarrow H_*(X) .\]
By taking the image of the constant constructible function $\one_X \in \Con(X)$ under this transformation, one obtains a class $c_{SM} (X) \in H_*(X)$, which coincides with the total Chern class for a smooth $X$. This condition, together with some functoriality properties, characterises the classes $c_{SM}(X)$ uniquely, making the class $c_{SM}$ a natural way to extend the the Chern class from smooth to singular varieties. The class constructed by MacPherson coincides with the one defined by Schwartz, and is called \textit{The Chern-Schwartz-MacPherson class} of $X$. This class was adapted by Kennedy  \cite{Ken} to an algebraic setting, where it takes values in the Chow ring $CH_*(X)$ instead of singular homology.

Motivic homotopy theory brings a setup which allows to deal with a wide array of homology theories on schemes. Each motivic ring spectrum $E \in \SH(k)$ defines a Borel-Moore homology, with the choice $E = H\mathbb{Z}$ defining Chow homology. It is possible to assign to each smooth scheme $X$ the Euler class of the tangent bundle of $X$, $e^E(X)$, which lives in the Borel-Moore homology that $E$ defines. In the case of Chow groups this is the top Chern class of $X$. With the class $e^E(X)$, being defined for a smooth $X$ as a starting point, one may seek to extend it to a class defined for every variety $X$. Such a class may be called \emph{'Euler-Schwartz-MacPherson class'}, and would be defined for an arbitrary motivic Borel-Moore homology. The Euler class are invariants satisfying a {\it cut-and-paste} relation with respect to subvarieties, i.e. they factor through the Grothendieck group of varieties.

 In section 2  of this paper we present the the motivic Euler class of the tangent bundle $e^E(X)$, defined for a motivic Borel-Moore homology theory $E$ and a smooth variety $X$; we prove the key property of additivity along blow-up squares. Section 3 contains a discussion on how to define characteristic classes for singular schemes in a general Borel-Moore homology theory. Replacing the group of constructible functions on $X$,  which appears in MacPherson's original work, by the simpler relative Grothendieck ring of varieties $K_0(\Var_{/X})$, we look for compatible functors
\[K_0(\Var_{/X}) \to E(X)\]
where $E(-)$ is a Borel-Moore homology functor. As a choice for a Borel-Moore homology $E(-)$, one possibility is to take the zeroth Borel-Moore homology defined by a motivic ring spectrum $E$, and then extend the Euler class $e^E(X)$, from smooth varieties to all varieties. We get a class
\[
e ^E \langle X \rangle \in E_0(X), \] in line with the construction of the class $mC_{-1}$ by Brasselet-Sch\"urmann-Yokura \cite{BSY} in $G$-theory, which can be recovered by the choosing the motivic ring spectrum $KGL$ representing algebraic $K$-theory. Another possibility for a Borel-Moore homology theory $E(-)$, is that of motivic pro-homology groups, containing data about the different completions (see Definition \ref{Defgoodclosure}) of each variety. These pro-homology groups are defined in section 4, where we also construct a {\it pro-Euler class} in these groups. The construction extends Aluffi's approach for constructing pro-CSM classes in the pro-Chow group setting, to a more general Borel-Moore pro-homology theory.  We construct a class as detailed in the following theorem.
\begin{Th}[Definition \ref{Defproeulerclass}]
	Let $k$ be a field of characteristic $0$ and $E \in \SH(k)$ a motivic ring spectrum. Then there is a well defined class \[ e^E\{X\}  \in \hat{E}_0(X/k)         \] for each algebraic variety $X$ over $k$. The definition extends the Euler class of the tangent bundle  $e^E(X) \in E_0(X/k)$ for smooth and proper varieties, to singular varieties $X$, and satisfies the {\it cut-and-paste} property. For $E = H \mathbb{Z}$, the class $e^{H \mathbb{Z}}\{-\}$ agrees with the top degree of the proCSM class constructed by Aluffi.
\end{Th}
 In our construction follow a similar approach to Aluffi, but in the generality of a motivic Borel-Moore homology theory. The main difference is that where Aluffi uses the Chern class of the sheaf of differential forms with poles along a divisor, we use an alternating sum of Euler classes of strata of a normal crossing divisor, as Chern classes do not exist in this generality.
On each good closure $\bar{U} / U$ (see Definition \ref{Defgoodclosure}), the class $e_{U}^{\bar{U}}$ in the motivic Borel-Moore homology group of $\bar{U}$, $E_0(\bar{U}/k)$, is defined by
\[
e_{U}^{\bar{U}} = \sum_{I \subset \mathcal{P}(\{1,...,r\})} (-1)^{|I|}i_{D_I *} e(D_I) ,
\]
where $D=\bar{U}\backslash U=\bigcup_{i=1}^{s}D_{i}$ is a simple normal
crossings divisor, $D_I = \cap_{i\in I} D_i$, $i_{D_I} : D_I \rightarrow \bar{U} $ are the inclusions, and $e(D_I)$ is the Euler class of the tangent bundle of $D_I$. The major part of the proof is showing that the collection of ${e^{\bar{U}}_U}$ gives a well defined class in the Borel-Moore pro-homology for $U$ smooth, that also extends to any variety $X$. The key property of the Euler class which makes this construction work, is additivity along blow-up squares, see Axiom \ref{Axiom}.

  Note that in a recent work, \cite{JSY}, Jin, Sun and Yang give independently a different construction of a pro-class in motivic Borel-Moore homology which also agrees with Aluffi's construction in pro-Chow groups. Their class agrees with the class $e^E\{X\}$ defined in this paper, see also Remark \ref{RemJSY}.
  
   In section 5 we deduce an extension of the motivic Gauss-Bonnet formula of \cite{DJK} to possibly singular proper varieties, using the pro-class described above. 
    \begin{Th}[Singular motivic Gauss-Bonnet formula, Theorem \ref{singGB}]
 	Let $\pi : X \rightarrow \Spec k$ be a proper variety over a field of characteristic $0$, and let $E \in \SH(k)$ be a motivic ring spectrum. Let $u: \one_k \rightarrow E $ be the unit map in $\SH (k)$. Then
 	\[ u_* \chic(X/k) = \pi_* e^E\{X\} .\]
 	\end{Th}
 
 As an example which gives a {\it quadratic refinement} for the classical classes in Chow groups, choose Eilenberg-Maclane spectrum of the Milnor-Witt sheaf, $E = H \mathcal{K}_*^{MW}$. Applying the constructions above we get a class in Chow-Witt groups $e^{CW}\langle X \rangle \in \widetilde{CH}_0(X)$ by our first definition, and a class in pro-Chow-Witt groups $e^{CW}\{X\} \in \hat{\widetilde{CH}}_0(X)$ by our second definition. If $X$ is proper, then pushing forward each of those classes to the base field we get the quadratic Euler characteristic with compact supports of $X$, $\chic(X/k)$. This is an identity of quadratic forms, i.e. in the Grothendieck-Witt ring of the base field, $\GW(k)$.

\paragraph*{Notation}
Let $k$ be a perfect field,  let $\Sch_k$ denote the category of separated finite type $k$-schemes, let $\Sch_k^{red}$ denote the full subcategory of reduced separated finite type $k$-schemes, we call objects in this category varieties. Let $\Sch_k'$ be the subcategory of $\Sch_k^{red}$ where only proper morphisms are considered. Let $\Sm_k$ denote the full subcategory of $\Sch_k$ of separated finite type smooth $k$-schemes. Let $\Sm^p_k$ denote the full subcategory of $\Sm_k$ spanned by schemes which are also proper over $k$. By a morphism of schemes we always mean a separated morphism of finite type. For a quasi-compact quasi-separated scheme $S$, we denote by $\SH(S)$ the stable $\infty$-category of motivic $S$-spectra. Throughout the paper we assume our field $k$ to admit resolution of singularities and weak factorization, see Appendix~\ref{section:resofsingandweakfac} for definitions.

\section*{Acknowledgements}

The author would like to thank Marc Levine for suggesting this topic in the course of the author's PhD studies, and for the many helpful discussions. The author would also like to thank Joseph Ayoub and Fangzhou Jin for very useful discussions. This paper is part of a project that has received funding from the European Research Council (ERC) under the European Union’s Horizon 2020 research and innovation programme (grant agreement No 832833). In addition the author is supported by the Swiss National Science Foundation (SNSF), project 200020 178729.\\
\includegraphics[scale=0.08]{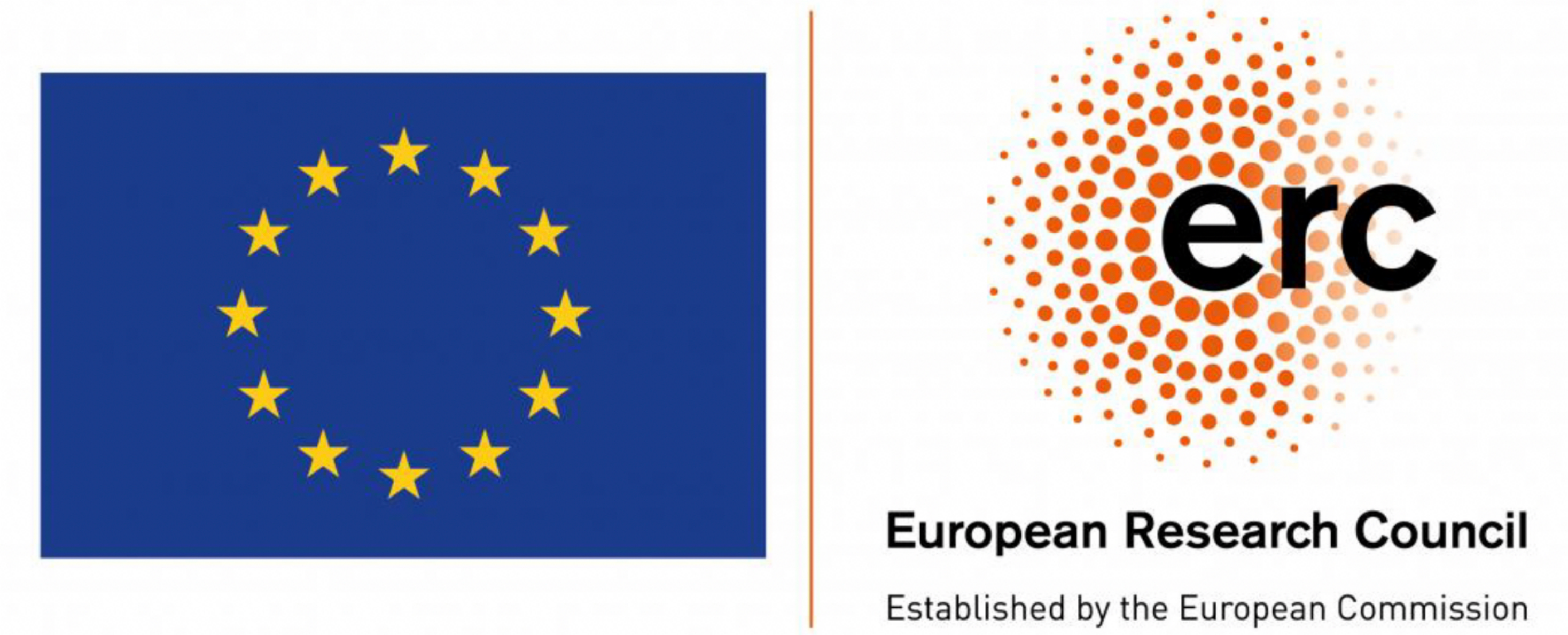}

\section{The motivic Euler class of the tangent bundle}

\subsection{$\mathbb{A}^1$-bivariant homology}

We use the formalism and notation of \cite{DJK} describing functoriality properties on motivic homology and cohomology theories.
Let $S$ be a quasi-compact and quasi-separated scheme and consider the symmetric monoidal category $\SH(S)$. For $X \in \Sm_S$ with a vector bundle $V \to X$ and zero section $s: X \to V$, we denote by $Th_X(V)$ the space $\Sigma_{\P^1}^\infty V /( V \setminus s(X))$ in $\SH(X)$. This definition can be extended to any virtual vector bundle on $X$, that is an element $v$ in the K-theory space $K(X)$, giving a $\otimes$-invertible object $Th_X(v) \in \SH(X)$, see \cite[2.1.5]{DJK}.

\begin{Def}
	Let $E\in SH(S)$ be a  motivic ring spectrum, let $p:X \to S $ be an $S$-scheme, with a virtual vector bundle {$v \in K(X)$}.
	Define the Borel-Moore homology (also called bivariant theory) spectrum with coefficients in $E$ on the scheme $X$ with twist $v \in K(X)$, to be the mapping spectrum 
	\[E(X/S, v) = \Maps_{\SH(X)}(Th_X(v), p^! E) ;\] 
	and the $n$-th Borel-Moore homology group by
	\[
	E_n(X/S,v)=\pi_nE(X/S, v) = [Th_X(v)[n],p^!E]_{\SH(X)} \simeq [p_!Th_X(v)[n], E]_{\SH(S)} .
	\]
	{We also set} $E_n(X/S)=E_n(X/S,0)$ and $H_n(X/S, v)=\one_n(X/S,v) $.
\end{Def}
For a proper morphism $f: X\rightarrow Y$, $X,Y \in \Sch_k$ there is the direct image map (\cite[2.2.7]{DJK}),
\begin{equation}
	f_*: E_n(X/k,f^*(v)) \rightarrow E_n(Y/k,v)  . \label{pforward}
\end{equation}
For a quasi-projective lci morphism $f: Y \rightarrow X $ we have the Gysin map \cite[4.3.4]{DJK}
\begin{equation}
	f^!: E_n(X/S, v) \rightarrow E_n(Y/S,f^*(v)+\langle L_f \rangle ) \label{pback} 
\end{equation} defined by a fundamental class $\eta_f \in E_n(Y/X, \langle L_f \rangle)$.
Here $L_f$ is the virtual vector bundle arising from the relative cotangent complex for $f$.
For a motivic ring spectrum $E$, the unit map $u : \one_S \to E$ defines the \textit{$\mathbb{A}^1$-regulator} natural transformation, \[ u_* : H_n(X/S, v) \rightarrow E_n(X/S, v) .\]

\subsection{The motivic Euler class}

\begin{Def} \label{Eulerclassdef}
	Let $E \in SH(S)$ be a motivic ring spectrum. Let $X \in \Sch_S$ and $V$ a vector bundle on $X$. Denote by $e^E(V)$ the Euler class of $V$, defined as the composite
	\[ \one_{X} \simeq \Sigma_{\P^1}^{\infty} V \to \Sigma_{\P^1}^{\infty} V/(V \setminus 0) = Th_X(V) \to Th_X(V) \otimes p^*E . \]
	In the case where $p:X \to S$ is smooth, we have $p^* \simeq Th_X(-T_p) \circ p^!$ (smooth pullback, see e.g. \cite[Theorem 6.18]{Hoy}), and therefore the Euler class defines a map $\one_X \to Th_X(V-T_p) \otimes p^! E $, hence a class in $E_0(X/S, T_p - V)$. If also $V=T_p$, the tangent bundle of $p : X \to S$, we get a class $e^E(T_p)$ in $E_0 (X/S)$ which we denote also by $e^E(X)$.	For more on the motivic Euler class see \cite[5.3]{BW}
\end{Def}

\begin{Rem}
	In the case $E = H\mathbb{Z} $ the class $e^{H\mathbb{Z}}(X) \in CH_0(X)$ is the Euler class of the tangent bundle in Chow groups, and that is the top Chern class of the scheme. The Chern-Schwarz-MacPherson class is an extension of the Chern class to a class defined for singular schemes. For other motivic ring spectra $E$ where we do not have a notion of Chern class, the Euler class of the tangent bundle $e^E(X)$ in the theory $E$ may still be considered. E.g., for $E= H \mathcal{K}^{MW}_*$, the Eilenberg-Maclane Milnor-Witt motivic spectrum, we get a class in Chow-Witt groups, which we denote by $e^{CW}(X) \in \widetilde{CH}_0(X)$, for more details see \cite[Section 2]{Le20}. Our upshot is to extend this class, defined for smooth schemes, to a class for singular schemes in Borel-Moore $E$-homology.
\end{Rem}

\subsection{Additivity of the Euler class on blow-up squares}

We follow the description of Jin and Yang in \cite{JY19}. By $\SH_c(S)$ we denote the full subcategory of $\SH(S)$ spanned by constructible objects. Given a scheme $f: Y \rightarrow \Spec k $, an object $M \in SH_c(X) $ and an endomorphism $ u: M \rightarrow M $, the characteristic class $C_Y(M,u) \in H_0(X/k, 0) $ is defined by the composition
\[  \one_X \xrightarrow{u} Hom (M,M) \rightarrow \mathbb{D}(M) \otimes M \simeq M \otimes \mathbb{D} (M) \rightarrow \mathcal{K}_Y \] where $ \mathcal{K} _Y := f^! \one_k $,  $ \mathbb{D} (M) := Hom (M, \mathcal{K}_Y) $. We define $ C_X(M):= C_X(M, id_M)$. \\
We describe this class explicitly in the case $ M=p_* \one_X $, where $S$ is a $k$-scheme, $X$ is a smooth $k$-scheme, and $ p: X \rightarrow S $ is a proper morphism, following \cite[Proposition 5.1.15]{JY19}. $C_X(p_* \one_X, u)$ is equal then to the composite
\[ 1_S \rightarrow p_*\one_X \xrightarrow{u} p_* \one_X \xrightarrow{p_*e(T_{X/k})} p_* \mathcal{K}_X \rightarrow \mathcal{K}_S  \label{C_X} \]
where the first map is the unit $id \to p_*p^*$ on $\one_S$, and the last map is the counit map $p_*p^! \simeq p_!p^! \rightarrow id $ applied on $\mathcal{K}_S=f^! \one_k $. We have then
\begin{equation}
	C_X(p_* \one_X) = p_* e(X) \in H_0 (S/k)
\end{equation}
where $p_*$ is the pushforward on motivic Borel-Moore homology.
We are about to use the following result in the proof of the theorem below.

\begin{Prop}  \label{fangzhu} [additivity of characteristic classes \cite[Corollary 5.1.5]{JY19}]  
	
	Given a distinguished triangle in $SH_{c}(X)$

	\[
	L\rightarrow M\rightarrow N\rightarrow
	\]
	The following formula holds
	\[
	C_{X}(M)=C_{X}(L)+C_{X}(N) .
	\]
\end{Prop}

\begin{Th}[additivity of Euler class on blow-ups] \label{Eulerblowup}
	Fix a motivic ring spectrum $A \in \SH(k)$ and denote by $e(\cdot)=e^A(\cdot)$ the Euler class. The following property is satisfied:
	let 
	\[
	\xymatrix{E\ar[d]\ar[r]^{\tilde{i}} & \tilde{X}\ar[d]^{q}\\
		C\ar[r]^{i} & X
	}
	\]
	be a blow-up square with $X$ and $C$ smooth over $k$, $X$ also proper. Then 
	\[
	e(X)-i_{*}e(C)=q_{*}(e(\tilde{X})-\tilde{i}_{*}e(E)) .
	\]	
\end{Th}

\begin{Rem}
	The property proven here is formulated as Axiom~\ref{Axiom} in the next section.
\end{Rem}	

\begin{proof}

	
	Consider the following diagram composed of pullback squares:
	\[
	\begin{tikzcd}E\ar[d, "r"]\ar[r, "\tilde{i}"] & \tilde{X}\ar[d, "q"] & \tilde{U}\ar[l, "\tilde{j}"']\ar[d, "\simeq"' , "s"]\\
		C\ar[r, "i"] & X & U \ar[l, "j"']
	\end{tikzcd}.
	\]
	Using the gluing property for motivic spectra (see e.g. \cite[Theorem 6.18 (4)]{Hoy}) on $\SH(X)$, i.e. the triangle 
	$$ j_{!}j^! \rightarrow id \rightarrow i_*i^* \rightarrow $$
	and applying it to $\one_X\in \SH(X) $ (for the open immersion $j$, $j^! = j^*$) we get
	$$	 j_{!}j^* \one_X \rightarrow \one_X \rightarrow i_*i^* \one_X \to $$, so we have
	\begin{equation}
		j_{!}\one_U \rightarrow \one_X \rightarrow i_*\one_C \to . \label{triangle1}
	\end{equation}
	Applying the same triangle to $ q_*\one_{\tilde{X}} \in \SH(X) $ gives
	$$ j_{!} j^* q_* \one_{\tilde{X}} \rightarrow q_* \one_{\tilde{X}} \rightarrow i_*i^*q_* \one_{\tilde{X}} \to $$
	Now $q_! = q_*$ as $q$ is proper, $s$ is an isomorphism, and by the smooth base change property (see e.g. \cite[Theorem 6.18 (3)]{Hoy}) on the square in the left we have
	$$  j_{!} j^* q_* \one_{\tilde{X}} \simeq j_{!} j^* q_! \one_{\tilde{X}} \simeq j_{!}s_! \tilde{j}^* \one_{\tilde{X}} \simeq j_{!}\one_U .$$
	By smooth base change in the square on the right, and using that for the closed immersions $i_! = i_*$,  $\tilde{i}_! = \tilde{i}_*$, we get
	$$ i_*i^*q_* \one_{\tilde{X}} \simeq i_*i^*q_! \one_{\tilde{X}} \simeq i_* r_! \tilde{i}^* \one_{\tilde{X}} \simeq i_* r_! \one_E \simeq  i_! r_! \one_E \simeq q_! \tilde{i}_! \one_E \simeq q_* \tilde{i}_* \one_E . $$
	In conclusion we have the triangle 
	\begin{equation}
		j_{!}\one_U \rightarrow q_* \one_{\tilde{X}} \rightarrow q_* \tilde{i}_* \one_E \to \label{triangle2} .
	\end{equation}
	Now we use proposition \ref{fangzhu} for the equations \eqref{triangle1} and \eqref{triangle2}. The objects in the triangles are all constructible since unit objects are constructible, and constructibility is stable under proper pushforward and exceptional pushforward, see \cite[Prposition 4.2.11, Proposition 4.2.12]{CD}. We get
	$$ C_X(j_{!}\one_U)=C_X(\one_X)-C_X(i_*\one_C) $$
	and
	$$ C_X(j_{!}\one_U)=C_X(q_* \one_{\tilde{X}})-C_X(q_* \tilde{i}_* \one_E) . $$
	By identity \eqref{C_X}, $C_{X}(q_{*}\tilde{i}_{*}\one_{E})=q_{*}\tilde{i}_{*}e(E)$,
	$C_{X}(q_{*}\one_{\tilde{X}})=q_{*}e(\tilde{X})$, $C_{X}(i_{C*}\one_{C})=i_{C*}e(C)$, and
	$C_{X}(\one_{X})=e(X)$, and so we obtain the desired formula.
\end{proof}

\section{Extending characteristic classes to singular schemes}

Sometimes a class defined on smooth varieties can be factored through the relative Grothendieck group of varieties $K_0(\Var_{/-})$. For a variety $S$,  $K_0(\Var_{/S})$ is the abelian group generated by isomorphism classes of morphisms $[X \to S]$, with the {\it cut-and-paste} relation $[X \to S] = [Y \to S] + [X - Y \to S ] $, where $Y \subset X$ is a closed $S$-subvariety.  Bittner's theorem \cite{Bitt} gives an equivalent {\it blow-up additivity} relation, which also defines the group. This yields a criterion for defining classes satisfying {\it cut-and-paste}, as noted by Brasselet, Sch\"urmann and Yokura in \cite[Corollary 0.1]{BSY}. We repeat this here, formulating it for an abstract Borel-Moore homology theory. Then we suggest two different ways to define a motivic Euler class based on that criterion.

\subsection{Borel-Moore Homology}

\begin{Def} \label{BMhomologyDef}
	Let $\Sch_k'$ be the subcategory of $\Sch^{red}_k$ with morphisms the proper maps. Let $\sA$ be an additive category such that filtered projective limits exist in $\sA$. 
	Let   $E:\Sch_k'\to \sA$ be a functor. For a morphism $f:Y\to X$ in $\Sch_k'$, denote by $f_*:E(Y)\to E(X)$ the morphism $E(f)$. We say that $E$ is an {\em additive} functor if for $X\in \Sch^{red}_k$ a disjoint union, $X=U\amalg V$, with inclusions $j_U:U\to X$, $j_V:V\to X$, the map $j_{U*}+j_{V*}:E(U)\oplus E(V)\to E(X)$ is an isomorphism. 
	\\
	A {\em  Borel-Moore homology functor over $k$ with values in $\sA$} is an additive functor  $E:\Sch_k'\to \sA$.
\end{Def}

\begin{Ex} \label{exmotBH}
	Let $E \in \SH (k)$ be a motivic ring spectrum and let $n \in \mathbb{Z}$, then $X\mapsto E_n(X/k)$ defines a Borel-Moore homology functor \cite{DJK}. We call such homology functor \emph{a motivic Borel-Moore homology functor}. 
\end{Ex}

\subsection{The blow-up additivity condition}

Our main method of producing characteristic classes for general varieties from classes defined for smooth varieties is based on the following condition which we refer to as Axiom \ref{Axiom}. We assume we are given a class $c(X)$ for every smooth variety and we would like to extend it to a well defined class on all varieties.

\begin{Axiom} \label{Axiom} Let $\{c(X) \in E(X)\}_{X \in \sC}$ be a family of classes defined for $\sC = \Sm^p_k$ or $\sC = \Sm_k$, where $E$ is a Borel-Moore homology functor. Then
	for each blow-up square 
	\[
	\begin{tikzcd}
		E\ar[d, "r"]\ar[r, hook, "j"] & \tilde{X}\ar[d, "q"]\\
		C\ar[r, hook, "i"] & X
	\end{tikzcd}
	\]
	with $X$ and $C$ smooth, $X$ proper over $k$, $\tilde{X} = Bl_C(X)$, $E = q^{-1}(X)$, we have
	\[
	c(X)-i_{*}c(C)=q_{*}(c(\tilde{X})-j_{*}c(E)).
	\]
\end{Axiom}
Note that by Theorem \ref{Eulerblowup}, this axiom is satisfied for the Euler class in motivic Borel-Moore homology.

\begin{Prop} \label{cclass} \cite[Corollary 0.1]{BSY}
	Let  $E:\Sch_k'\to \sA$ be a Borel-Moore homology functor as in \ref{BMhomologyDef}. Assume that we fix for each smooth $X \in \Sm_k$ a class $c(X) \in E(X)$, such that
	for each isomorphism $h: X \to X'$, $h_*(c(X))=c(X')$, and the blow-up axiom \ref{Axiom} is satisfied. Then $c(-)$ can be extended to a system of group homomorphisms
	\[ c_X\{-\}: K_0(\Var_{/X}) \to E(X)\] for each $X \in \Sch_X$, that satisfy:
	\item
	$c\{X\}:=c_X\{id_X :X \to X \}=c(X)$ for each smooth $X$ ;
	\item
	$c_X\{f:Y \to X\} = f_* c(Y)$ for $Y$ smooth over $k$, and for each proper $f: Y \to X$. \\
	These conditions determine the homomorphisms $c_X\{-\}$ uniquely.
\end{Prop}

\begin{proof}
	This is a consequence of Bittner's theorem \cite[Theorem 5.1]{Bitt}, characterising the group $K_0(\Var_{/S})$ as the group generated by $S$-varieties $X$ which are smooth over $k$ and proper over $S$, with relations given by additivity with respect to blow-ups of $X$ over smooth centres. The additivity conditions are the same as the squares in Axiom~\ref{Axiom} . Then for each $S \in \Sch_k^{red}$,
	\[ [X \xrightarrow{f} S] \mapsto f_*c(X) \] defines $c_S\{-\}$. 
\end{proof}
 A straightforward property is compatibility with pushforwards, which can be considered a weak version of MacPherson's natural transformation, replacing constructible functions $\Con(-)$ by the relative Grothendieck ring of varieties $K_0(\Var_{/-})$.
\begin{Prop}
	With notation as in Proposition~\ref{cclass} above, let $p: Z \to X$ be a proper map of varieties, $X, Z \in \Sch_k^{red}$. The following diagram commutes
	\[
	\begin{tikzcd} K_0(\Var_{/Z}) \ar[r, "c_Z\{-\}"]\ar[d, "p_*"] & E(Z)\ar[d, "p_*"] \\
		K_0(\Var_{/X}) \ar[r, "c_X\{-\}"] & E(X) .
	\end{tikzcd}
	\]
	
\end{Prop}

\begin{proof}
	Since they generate $K_0(\Var_{/Z})$, it is enough to check the claim for elements of the form $[f: W \to Z]$ with $f$ proper and $W$ smooth over $k$. Then $c_Z\{W\} = f_* c(W)$, and so $p_* c_Z\{W\} =p_* f_* c(W)$. On the other hand, $p_* [f: W \to Z] = [p \circ f : W \to X] \in K_0(\Var_{/X})$, and then $c _X\{p_* [f: W \to Z]\} = (p \circ f)_* c(W) = p_*f_* c(W)$ as well, so the diagram commutes.
\end{proof}

\subsubsection*{Definition I - The singular Euler class} \label{Eulerclassex1}
	Let $E_0(-/k)$ be the Borel-Moore homology functor defined by a motivic ring spectrum $E \in \SH(k)$ (see equation \eqref{pforward}). We have the Euler classes $e^E(X) \in E_0(X/k)$ for each $X \in \Sm_k$ (Definition \ref{Eulerclassdef}). By Theorem~\ref{Eulerblowup} those classes satisfy Axiom \ref{Axiom}. Thus they extend to define a class $e ^E \langle X \rangle := c_X\{[id: X \to X]\}$ in $E_0(X/k)$ for all $X \in \Sch_k^{red}$.
\begin{Ex} Take the motivic spectrum representing $K$-theory, $E=KGL$. We get then a class in $KGL_0(X/k)=G_0(X)$ (see e.g. \cite[Theorem 1.0.12]{Jin}), which for a smooth $X$ is the $K$-theoretical Euler class. By definition the class $e^{KGL} \langle X \rangle$ we obtain for a variety $X\in \Sch^{red}_k$ is the same as the zero degree term of the characteristic class $mC_{-1}$ in $G$-theory, defined by Brasselet, Sch\"urmann and Yokura in \cite[Theorem 2.1]{BSY}.
\end{Ex}

\begin{Ex} \label{exquadraticclass} Taking $E = H \mathbb{Z}$ we get a class in Chow groups, $e^{H \mathbb{Z}}\langle X\rangle \in CH_0(X)$.
	Taking $E=H \mathcal{K}^{MW}$, the Eilenberg-Maclane spectrum of Milnor-Witt Theory, we get a class 
	$e \langle X \rangle \in $ $H \mathcal{K}^{MW}_0(X/k) = {\widetilde{CH}}_0(X)$ (for the group equality see \cite[Chapter 6, Corollary 4.2.13]{BCD}),  in the zeroth Chow-Witt homology group. This can be considered as a \emph{quadratic refinement} to the Euler class in Chow groups $e^{H \mathbb{Z}}\langle X \rangle$. However, we also have the more elaborate construction of a pro-class in Example \ref{exproquadraticclass}. When pushed forward to the base field we get an element of $\GW(k)$, which yields the quadratic Euler characteristic of $X$, see Theorem \ref{singGB}.
\end{Ex}

\begin{Rem} The Chern-Schwartz-MacPherson class is compatible with specialization, see \cite{Ver}, \cite{Ken2}, and \cite[Theorem 3.4]{Schu}.	A similar specialization statement involving the class $e\langle X \rangle$ in our setting, at least for orientable theories, will appear in a subsequent work.
\end{Rem}
\subsubsection*{Definition II - The singular pro-Euler class} \label{Eulerclassex2}
	Let $E \in \SH(k)$ be a motivic ring spectrum. Following Aluffi, in Section \ref{sec:pro} we construct a Borel-Moore homology functor in pro-groups $\hat{E}_0(-/k)$, taking into account the inverse system of the different possible ways to complete a variety. We then define what we call \emph{the pro-Euler class} ${e} ^E\{X\} \in \hat{E}_0(X/k)$ for each $X \in \Sm_k$. Theorem~\ref{cdeftheorem} shows that the same condition as in Proposition \ref{cclass}, Axiom \ref{Axiom}, is what allows to construct such well defined class ${e}\{X\} \in \hat{E}_0(X/k)$ for each $X \in \Sch^{red}_k$. This class also factors through $K_0(\Var_{/-})$ as in Proposition  \ref{cclass}.
	
\begin{Ex}
	Take $E=H \mathbb{Z}$, the motivic Eilenberg-Maclane spectrum and let $X \in \Sch_k^{red}$. Then with $e^{H \mathbb{Z}}\{X\} \in \hat{H \mathbb{Z}}_0(X/k)= \hat{CH}_0(X)$ we get the same class as the zero degree of Aluffi's pro-CSM class $\{X\}_0 \in \hat{CH}_0(X)$, see Remark \ref{Alufficomp}. 
\end{Ex}

\begin{Ex} \label{exproquadraticclass}
	Take $E=H \mathcal{K}^{MW}$, the Eilenberg-Maclane spectrum of Milnor-Witt Theory, and let $\pi: X\to k$ be and algebraic variety.  we get $e^{H \mathcal{K}^{MW}}\{X\} \in \hat{H \mathcal{K}}^{MW}_0(X/k) = \hat{\widetilde{CH}}_0(X)$ in the pro-Chow-Witt group  $\hat{\widetilde{CH}}_0(X)$. That is a \emph{quadratic refinement} of the pro-Chow Euler class $e^{H \mathbb{Z}}\{X\}$. If we push it forward to the base field, i.e., consider $\pi_*e^{H \mathcal{K}^{MW}}\{X\}$, we get an expression in quadratic forms, an element of $\hat{\widetilde{CH}}_0(k/k)=\GW(k)$. This is the quadratic Euler characteristic of $X$ as we see in Theorem \ref{singGB}.
\end{Ex}

\begin{Rem}
	For $X$ proper, the Borel-Moore pro-homology groups in \emph{Definition II} agree with the ones in \emph{Definition I} (see Lemma \ref{lemmaprogroups}), and then also the two definitions of the Euler class coincide.
	However, for a non-proper $X$, the two definitions of Euler class take values in different groups.  The motivic Gauss-Bonnet formula we prove in section~\ref{section:GaussBonnet} is valid for both definitions.
\end{Rem}

\section{The pro-Euler class} \label{sec:pro}

\subsection{Motivic pro-homology}

We define here pro-homology groups for a bivariant theory, and describe how to get well-defined classes in them. {The discussion here is identical to \cite[sections 2,3]{Alu}, where here we use an arbitrary Borel-Moore homology theory $E(-)$ instead of Chow groups; the existence of a good theory of proper pushforward for $E$-Borel-Moore homology allows us to use Aluffi's arguments in this more general setting}.

\subsubsection*{Closures and completions}

We repeat here the definitions as in Aluffi \cite[Section 2]{Alu}.

\begin{Def} Let $X$ be in $\Sch_k^{red}$.\\[5pt] \label{Defgoodclosure}
	1. A morphism $i:X\hookrightarrow Z^{i}$ in $\Sch_k$ with $Z^{i}$ proper over $k$ and $i$ an open immersion is called a {\em completion} of $X${; if $i(X)$ is dense in $Z^i$, we call $i$ a {\em closure} of $X$.}  If  $i:X \hookrightarrow Z^i$ is a closure of $X$ such that $Z^i$ is smooth over $k$ and the complement $Z^i\setminus i(X)$ is a simple normal crossing divisor (see Definition \ref{Defncd}), we call $i$ a {\em good closure} of $X$. \\[2pt]
	2.  Let $\mathcal{C}^*(X)$ be the category with objects the completions  $i:X\rightarrow Z^{i}$ and where a morphism $\pi:i\to j$ is a proper morphism $\pi:Z^i\to Z^j$ such that the diagram
	\[
	\begin{tikzcd} [row sep = tiny, column sep = large]
		& Z^{i} \arrow[dd , "\pi"]\\
		X \arrow[dr, "j"] \arrow[ur ,"i"] & \\
		& Z^{j}
	\end{tikzcd}
	\]
	commutes. Let $\mathcal{C}(X)$ be the full subcategory of $\mathcal{C}^*(X)$ with objects the closures of $X$ and  let $\mathcal{C}^g(X)$ be the full subcategory of $\mathcal{C}(X)$ with objects the good closures of $X$.\\[2pt]
	3. Let $f:Y\to X$ be a morphism in $\Sch_k$, let $i:X\to Z^i$, $j:Y\to W^j$ be closures of $X$ and $Y$, respectively. A morphism $\pi:W^j\to Z^i$ is a {\em morphism of closures over $f$} if 
	the diagram
	\[
	\xymatrix{ 
		Y\ar[r]^j\ar[d]_f&W^j\ar[d]^\pi \\
		X\ar[r]^{i}   & Z^{i}
	}
	\]
	commutes. 
\end{Def}

\begin{Lemma} \label{cofinal}
	1. $\mathcal{C}^*(X)$ is left-filtering and $\mathcal{C}(X)\subset \mathcal{C}^*(X)$ is a left cofinal subcategory.\\[2pt]
	2. For each pair of closures $i:X\rightarrow Z^{i}$, $j:X\rightarrow Z^j$ there is at most one morphism $\pi:i\to j$ in $\mathcal{C}(X)$.\\[2pt]
	3. Let $f:Y\to X$ be a morphism in $\Sch_k$ and let $i:X\to Z^i$ be a closure of $X$. For $j:Y\to W^j$ a closure of $Y$, there is at most one morphism of closures $\pi:W^j\to Z^i$ {over $f$}. Moreover, the subcategory  $\mathcal{C}(Y)/(i,f)$ of  $\mathcal{C}(Y)$ with objects the closures  $j:Y\to W^j$ such that there exists a morphism of closures {$j\to i$ over $f$} is left cofinal in $\mathcal{C}(Y)$.\\[2pt]
	4. Suppose that $k$ admits resolution (see Definition \ref{resolutionsing}) of singularities and $X$ is smooth over $k$.  Then (1), (2) and (3) hold with 
	$\mathcal{C}(X)$ replaced by  $\mathcal{C}^g(X)$.
\end{Lemma}

\begin{proof} The compactification theorem of Nagata (see e.g. \cite[Theorem 38.33.8]{stacks-project} ) implies there exists a closure of $X$ {so $\mathcal{C}(X)$ and $\mathcal{C}^*(X)$ are non-empty}. {Given completions $i:X\to Z^i$,  $j:X\to Z^j$ and two morphisms $f_1, f_2:i\to j$, let 
		$Z^h\subset Z^i$ be the closure of $i(X)$, with $h:X\to Z^h$, $g:Z^h\to Z^i$ the induced morphisms. Then $h$ is a closure of $X$ and $g$ defines a morphism of completions $g:h\to i$.}
	Since $X$ is dense in $Z^h$ and $Z^j$ is separated over $k$, we have $f_1\circ g=f_2\circ g$. This proves (1). A similar argument proves (2) and gives the uniqueness assertion in (3). Given $f:Y\to X$, $i:X\to Z^i$ and $j:Y\to W^j$ as in (3), let $W^h$ be the closure of $(j, i\circ f)(Y)$ in $W^j\times_k Z^i$, then the induced map $h:Y\to W^h$ is a closure of $Y$ with morphism $p_2:W^h\to Z^i$ of closures over $f$ and morphism $p_1:W^h\to W^j$ of closures of $Y$. This proves the second part of (3).
	
	For (4), resolution of singularities says that for each closure $i:X\to Z^i$ of $X$ there is a good closure $j:X\to Z^j$ and a morphism of closures $\pi:i\to j$. Then assertion (4) follows from this together with (1)-(3).
\end{proof}
\subsubsection*{Motivic Borel-Moore pro-homology} 


 We define the pro-homology group of $X$ as the limit of inverse system of Borel-Moore groups of the completions, in a similar fashion to the way Aluffi defines pro-Chow groups.

\begin{Def}  Let  $E:\Sch_k'\to \sA$ be a Borel-Moore homology functor and let $X\in \Sch_k^{red}$. Sending a completion $i:X\to Z^i$ to $E(Z^i)$ and a morphism $\pi:(j:X\to Z^j)\to (i:X\to Z^i)$ to the map $\pi_*:E(Z^j)\to E(Z^i)$ defines a functor
	\[
	E(-):\mathcal{C}^*(X)\to \mathcal{A}.
	\]
	Define the Borel-Moore $E$-pro-homology group of $X$ by 
	\[
	\hat{E}(X)=\lim_{i\in\mathcal{C}^*(X)}E(Z^{i}).
	\]
\end{Def}

It is enough to consider only (good) closures instead of all completions; we have naturally defined proper pushforward maps on pro-homology groups, making this a Borel-Moore homology theory in the sense of Definition \ref{BMhomologyDef}, as is assured by the following lemma.

\begin{Lemma}  \label{lemmaprogroups} 1. The canonical map 
	\[
	\lim_{i\in\mathcal{C}(X)}E(Z^{i})\to \hat{E}(X)
	\]
	is an isomorphism. If $X$ is smooth over $k$, and $k$ admits resolution of singularities, the canonical map 
	\[
	\lim_{i\in\mathcal{C}^g(X)}E(Z^{i})\to  \lim_{i\in\mathcal{C}(X)}E(Z^{i})
	\]
	is an isomorphism as well.\\
	2. Given any morphism $f:Y\to X$ in $\Sch_k$, the morphisms $\pi_*:E(W^{j})\to E(Z^{i})$ for each morphism $\pi$ of $j:Y\to W^j$ to $i:X\to Z^i$ give rise to a well-defined morphism
	\[
	f_*:=\hat{E}(f):\hat{E}(Y)\to \hat{E}(X).
	\]
	Moreover, we have $(fg)_*=f_*g_*$ for composable morphism $g:W\to Y$, $f:Y\to X$. \\
	3. If $X$ is proper over $k$, there is a canonical isomorphism $\hat{E}(X) \cong E(X) $. \\
		4. If $f : Y \to X$ is a proper map of proper schemes over $k$, the induced homomorphism
	\[
	f_* : E(Y) \cong \hat{E}(Y) \longrightarrow \hat{E}(X) \cong E(X)
	\]
	is the usual pushforward in the Borel-Moore homology functor $E$. 
\end{Lemma}
\begin{proof}For 1. and 2. use Lemma~\ref{cofinal}. For 3. and 4. note that when $X$ is proper over $k$, $id:X \to X$ is an initial element in the category $\sC (X)$.
	\end{proof}

Let now $E \in \SH(k)$ be a motivic ring spectrum. We consider the motivic Borel-Moore functor $ X \mapsto E_*(X/k)$ defined by $E$ as in Example \ref{exmotBH}. By the above construction we get a new Borel-Moore homology functor
\[
X \mapsto \hat{E}_*(X) \] which we call the {\it motivic Borel-Moore pro-homology functor}.
\begin{Ex}
	 For $E = H \mathbb{Z}$,  using the fact that $H \mathbb{Z} _n (X/k) = CH_n(X/k)$, $\hat{E}_n(X) = \hat{CH}_n(X) $ give the proChow groups defined by Aluffi \cite[Def. 2.2]{Alu}.
\end{Ex}

	

\subsection{Motivic pro-Euler class}
	
		\subsubsection*{Outline} \label {proclass}
	
	We describe here how to get a well defined class $c\{X\}$ in Borel-Moore pro-homology for every variety $X$, and we show that in fact all we need is the classes on smooth proper varieties to satisfy Axiom \ref{Axiom}. 	We start here from recalling our purpose, and going the reverse way to find what is required from the classes on smooth varieties we begin with. The first two steps are similar to Aluffi's construction in \cite{Alu} and we give the details in the appendix. Step 3 which is special to our construction is detailed below in the rest of this section.
\begin{description}
	\item[Step I:] Getting a well-defined pro-class $c\{X\}$ for any variety $X$, from well defined classes $c\{U\}$ for smooth varieties $U$. Since for each such $X$ we have a stratification $X = \amalg U_i$, with the $U_i$ smooth, we must have $c\{X\} = \sum_{i\in I}\alpha_{i*}(c\{U_i\})$, and the class must not depend on the stratification.	To ensure that, $c\{U\}$ has to satisfy the cut and paste property, as formulated in Proposition \ref{classvariety} (and so, to factor through the Grothendieck ring of varieties).
\item[Step II:] Getting a pro-class $c\{U\} \in \hat{E}(U)$ for a smooth $U$ from a family of classes $c_U^{\bar{U}} \in E(\bar{U})$ for good closures $U \hookrightarrow \bar{U}$. The condition to ensure that is called \emph{good local data} as in \cite{Alu} and is formulated in Proposition \ref{goodlocaldata}.
\item[Step III:] Defining $c_U^{\bar{U}}$ as alternating sums of classes for smooth strata in $\bar{U}$, when we have classes $c(U)$ for smooth and proper $U$; then showing that this defines a good local data, given  Axiom \ref{Axiom}. This is detailed in the rest of the section; see Theorem \ref{cdeftheorem} below.
\end{description}	
	Since the motivic Euler class for smooth proper $U$, $e^E(U)$ satisfies Axiom \ref{Axiom}, as we saw in Theorem \ref{Eulerblowup}, we obtain this way a motivic pro-Euler class $e^E\{X\} \in \hat{E}_0(X/k)$, which is the main construction of this paper.

	\subsubsection*{Alternating sums of classes of the strata}
	
	We start with given classes $c(U)$ for each smooth and proper $U$ over $k$. We show that our condition, Axiom~\ref{Axiom}, ensures that those classes extend to define a class $c\{X\}$ for each $X\in \Sch^{red}_k$.

	\begin{Def} \label{cdef} Let $\Sm^p_k\subset \Sm_k$ be the full subcategory of smooth and proper $k$-schemes.  Suppose we are given an element  $c(X) \in E(X/k)$ for each  $X\in \Sm^p_k$. We suppose in addition that  if $X=X_1\amalg X_2$ is a disjoint union of open subschemes with inclusion maps $i_j:X_j\to X$, then   $c(X)=i_{1*}(c(X_1))+i_{2*}(c(X_2))$. We call such a family $\{c(X)\in E(X/k)\}_{X\in \Sm^p_k}$ an {\em additive family of characteristic classes for smooth, proper $k$-schemes.}
		
		Let $(U,\bar{U})$ a pair of $U \in \Sch_k$ and a good closure $i: U \hookrightarrow \bar{U}$.
		{Let $D_1,\ldots, D_r$ be the irreducible components of $D$.  For $\emptyset \neq I \subset \{1,...,r\}$, define} $D_I = \cap_{i\in I} D_i$, with closed immersion $ i_{D_I} : D_I \hookrightarrow \bar{U}$. Set $D_\emptyset = \bar{U}$ {and let $|I|$ denote the cardinality of $I$}. We define
		\[
		c_{U}^{\bar{U}}=\sum_{I \subset \mathcal{P}(\{1,...,r\})} (-1)^{|I|}i_{D_I *} c(D_I)  \in E(\bar{U}).
		\]
		\end{Def}
	
	\begin{Lemma}  \label{lemma1}
		{Let $\{c(X)\in E(X/k)\}_{X\in \Sm^p_k}$ an additive family of characteristic classes for smooth projective  $k$-schemes, and form the classes $c^{\bar{U}}_U\in E(\bar{U})$ following Definition~\ref{cdef}.  Take $V\in \Sm_k$ and let $V \hookrightarrow \bar{V} $ be a good closure. Let $\tilde{D}=\bar{V}\setminus V$. Let $F \subset \bar{V}$ be a smooth divisor such that $F+\tilde{D}$ is a simple normal crossing divisor on $\bar{V}$; equivalently, $F$ intersects $\tilde{D}$ with normal crossing.  Let $E=V\cap F$.  Then $\bar{V}$ is a good closure of $V\setminus E$,  $F$ is a good closure of $E$, and we have
			\[
			c_{V\setminus E}^{\bar{V}} = c_V^{\bar{V}} - i_* c_{E}^F .
			\] }
	\end{Lemma}
	
	\begin{proof}
		{We have $\bar{V} \setminus (V \setminus E) = F+\tilde{D}$, a simple normal crossing divisor. Similarly, $F\setminus E=F\cap \tilde{D}$, and since $F$ intersects $\tilde{D}$ with normal crossing, $F\cap \tilde{D}$ is a simple normal crossing divisor on $F$. Thus $\bar{V}$ is a good closure of $V\setminus E$ and $F$ is a good closure of $E$.}
		
		{Suppose $\tilde{D}$ has irreducible components $\tilde{D}_1,\ldots, \tilde{D}_r$. We consider the smooth divisor $F$ as the ``component'' $G_0$ of the simple normal crossing divisor $G:=F+\tilde{D}$, and separate the strata $G_I$ of $G$ into those which involve $F$ and those which do not. This gives
			\[
			c_{V\setminus E}^{\bar{V}}=\sum_{I \subset \{1,...,r\}} (-1)^{|I|}i_{\tilde{D}_I *} c(\tilde{D}_I)+ 
			\sum_{I \subset  \{1,...,r\}} (-1)^{|I|+1}i_{\tilde{D}_I\cap F *} c(\tilde{D}_I \cap F)=
			c_V^{\bar{V}} - i_* c_E^F.
			\]
			The additivity of the classes $c(-)$ with respect to disjoint union justifies the  first identity.}
		\end{proof}
	The following Theorem justifies our main construction, The proof uses results and notions from the appendix.
	
	\begin{Th} \label{cdeftheorem}{Suppose that $k$ admits resolution of singularities and weak factorization. Let $\{c(X)\in E(X/k)\}_{X\in \Sm^p_k}$ be a family of characteristic classes for smooth, proper $k$-schemes, satisfying Axiom~\ref{Axiom}. Then \\
			1. For each $U\in \Sm_k$, there is a unique element $c\{U\}\in \hat{E}(U/k)$ such that, for each $\bar{U}\in \mathcal{C}^g(U)$, the  $(U,\bar{U})$-component of $c\{U\}$  is $c_U^{\bar{U}}\in E(\bar{U}/k)$, the class given by definition~\ref{cdef}.\\
			2. For each $X\in \Sch_k^{red}$, there is a unique element $c\{X\}\in \hat{E}(X/k)$ such that for each smooth stratification $\sU=\{\alpha_i:U_i\to X\mid i\in I\}$ of $X$, we have
			\[
			c\{X\}=\sum_{i\in I}\alpha_{i*}(c\{U\})\in \hat{E}(X/k).
			\]}
	\end{Th} 
	
	\begin{proof}
		{To prove (1) and (2), it suffices by Proposition~\ref{goodclass} to show that the family of classes $\{ c_U^{\bar{U}} \in E(\bar{U}/k)\}_{U\in \Sm_k, \bar{U}\in \mathcal{C}^g(U)}$ forms good local data, as in Definition \ref{goodlocaldata}. Retaining the notation of that definition, we need to verify the additivity property \eqref{additiivity} and the identity \eqref{gld}.}
		
		First, by satisfying Axiom \ref{Axiom}, the family $\{c(X)\in E(X/k)\}_{X\in \Sm^p_k}$ is an additive family in the sense defined above, since By Bittner's theorem this axiom implies that the class factors through the Grothendieck group of varieties. Then the additivity property  \eqref{additiivity} follows immediately.
		
		To verify \eqref{gld}, let $(U, \bar{U}, W)$ be a triple as in Definition~\ref{goodlocaldata}, let $\pi:\bar{V}\to \bar{U}$ be the blow-up $Bl_W\bar{U}$ with exceptional divisor $F\subset \bar{V}$, let $w:W\to \bar{U}$, $f:F\to \bar{V}$ be the inclusions, let $Z=U\cap W$, $V=\pi^{-1}(U)$, and let $E=V\cap F$. Let $D$ be the simple normal crossing divisor $\bar{U}\setminus U$ on $\bar{U}$ and let $\tilde{D}$ be the proper transform $p^{-1}[D]$. 
		
		{Let $D_1,\ldots, D_r$ be the irreducible components of $D$ and let $\tilde{D}_i$ be the proper transform $\pi^{-1}[D_i]$.  For $I\subset \{1,\ldots, r\}$ we have the strata $D_I:=\cap_{i\in I}D_i$ and $\tilde{D}_I:=\cap_{i\in I}\tilde{D}_i$.  Let $w_I:W\cap D_I\to D_I$, $f_I:F\cap \tilde{D}_I\to \tilde{D}_I$ be the inclusions and let $\pi_I:\tilde{D}_I\to D_I$ be the restriction of $\pi$.  We claim that
			\begin{equation}\label{BlowUpIdentity}
				c(D_I) - w_{I*} c(D_I \cap W) = \pi_{I*} (c(\tilde{D}_I) - f_{I*} c(\tilde{D}_I \cap F)) 
			\end{equation}
			for each $I\subset  \{1,\ldots, r\}$. Indeed, for $I=\emptyset$, we have the blowup square}
		\[
		\begin{tikzcd}
			F\ar[d, "r"]\ar[r, hook, "j"] & \bar{V}\ar[d, "\pi"]\\
			W\ar[r, hook, "w"] & \bar{U}
		\end{tikzcd}
		\]
		{and since $D_\emptyset=\bar{U}$, $\tilde{D}_\emptyset=\bar{V}$ $\pi_\emptyset=\pi$, $f_\emptyset=f$,  the identity follows from Axiom~\ref{Axiom} applied to this blowup square. For non-empty $I$, the fact that $W$ intersects $D$ with normal crossing   implies that }
		\[
		\begin{tikzcd}
			\tilde{D}_I \cap F \ar[d, "r"]\ar[r, hook, "j"] &\tilde{D}_I \ar[d, "\pi"]\\
			D_I \cap W\ar[r, hook, "w"] & D_I
		\end{tikzcd}
		\]
		{is also a blowup square. Thus  Axiom~\ref{Axiom} applied to this blowup square gives the identity \eqref{BlowUpIdentity}. }
		
		{Since $\bar{V}$ is a good closure of $V\setminus E$ with $\tilde{D}+F=\bar{V}\setminus(V\setminus E)$, we have
			\[
			c_V^{\bar{V}}=\sum_{I \subset  \{1,...,r\}} (-1)^{|I|}i_{\tilde{D}_I*} c(\tilde{D}_I),
			\]
			where $i_{\tilde{D}_I}:\tilde{D}_I\to \bar{V}$ is the inclusion. Similarly,
			\[
			c_E^F=\sum_{I \subset  \{1,...,r\}} (-1)^{|I|}i_{\tilde{D}_I\cap F*} c(\tilde{D}_I\cap F),
			\]
			where $i_{\tilde{D}_I\cap F}:\tilde{D}_I\cap F\to F$ is the inclusion,
			\[
			c_U^{\bar{U}}=\sum_{I \subset  \{1,...,r\}} (-1)^{|I|}i_{D_I*} c(D_I),
			\]
			where $i_{D_I}:D_I\to \bar{U}$ is the inclusion and
			\[
			c_Z^W=\sum_{I \subset  \{1,...,r\}} (-1)^{|I|}i_{D_I\cap W*} c(D_I\cap W), 
			\]
			where $i_{D_I\cap W}:D_I\cap W\to W$ is the inclusion.}
			{Applying Lemma~\ref{lemma1} to the blowup square for $D_I \cap W\subset D_I$ gives the identity 
			\[
			\pi_{I*}(c(\tilde{D}_I)-f_{I*}(c(\tilde{D}_I\cap F) )=
			c(D_I)-w_{I*}(c(D_I\cap W))
			\]
			for each $I\subset \{1,\ldots,r\}$. Using the functoriality of pushforward, we thus have
			\begin{align*}
				\pi_*(c_V^{\bar{V}}-f_*(c_E^F))& =  
				\sum_{I \subset  \{1,...,r\}} (-1)^{|I|}i_{D_I*}( \pi_{I*}(c(\tilde{D}_I)-f_{I*}(c(\tilde{D}_I\cap F) )\\
				& =  
				\sum_{I \subset  \{1,...,r\}} (-1)^{|I|}i_{D_I*}(c(D_I)-w_{I*}(c(D_I\cap W))\\
				&=c_U^{\bar{U}}-w_* c_Z^W.
		\end{align*}}
		This verifies the identity \eqref{gld}.  
	\end{proof}
	
	

	\subsubsection*{The pro-Euler class} \label{ssec:proeulerclass}
	
	\begin{Def} \label{Defproeulerclass}
		Let $X \in \Sch_k^{red}$ and let $E \in \SH(k)$ a motivic ring spectrum. We define the pro-Euler class of $X$,  $e^E\{X\}$, as the unique pro-class in $\hat{E}_0(X)$ extending, according to Theorem \ref{cdeftheorem}, the Euler class of the tangent bundle $e^E(X) \in E_0(X/k) \simeq \hat{E}_0(X)$ for smooth and proper schemes. The condition of Theorem \ref{cdeftheorem} is satisfied by Theorem~\ref{Eulerblowup}. We also have the compatible maps on Grothendieck groups according to Proposition \ref{cclass}.
	\end{Def}

\begin{Rem} \label{Alufficomp}
	Definition \ref{Defproeulerclass} above gives a generalization of the class defined by Aluffi in \cite[Definition 4.1, Definition 4.4]{Alu} to motivic Borel-Moore homology. Let $E = H\mathbb{Z}$, then for $X$ smooth and proper our class $e^{H \mathbb{Z}}\{X\} \in \hat{E}_0(X/k) \simeq E_0(X/k) = CH_0(X/k)$ is the Euler class of the tangent bundle, which the is top Chern class, the same as Aluffi's class. For other $X$, the class $c_X^{\bar{X}} \in \hat{CH}_0(X/k)$ for a good compactification is the same as the one defined by Aluffi due to the cut-and-paste property that both classes satisfy (see \cite[Propsition 3.2]{Alu}), which makes the extension for all varieties unique. Therefore the class $e^{H \mathbb{Z}}\{X\}$ agrees with Aluffi's $\{X\}_0$ (\cite[Definition 4.4]{Alu}) for any $X \in \Sch_k^{red}$.
\end{Rem}

	

	\begin{Rem} \label{RemJSY}
		In the recent work \cite{JSY} Jin, Sun and Yang give an alternative approach for constructing a motivic pro-CSM class. The class they construct, $\hat{C}^c_X(\one_X)$, agrees with our class $e\{X\}$ for $X$ smooth. Then by Proposition \ref{cclass} the class $\hat{C}^c_{X}(X)$ agrees with our definition of $e\{X\}$ for any $X$. They prove independently that this class agrees with Aluffi's class on pro-Chow groups, see \cite[Theorem 3.4]{JSY}.
	\end{Rem}

\section{A singular Gauss-Bonnet formula} \label{section:GaussBonnet}

Here we define the motivic Euler characteristic with compact support, and prove how to extend the motivic  

\begin{Def} Let $(\mathcal{C}, \otimes, \one_\mathcal{C} )$ be a symmetric monoidal {category.}  Let $x$ be a strongly dualisable object of $\mathcal{C}$, with $x^\vee$ its dual, and $ \delta_x: \one_\mathcal{C} \rightarrow x \otimes x^\vee$, {$ev_x:  x^\vee \otimes x \rightarrow \one_\mathcal{C}  $  } the co-evaluation and evaluation maps respectively. 
		For $f:x\to x$ an endomorphism, the {\em trace} of $f$ is the element $tr(f)\in \End_\mathcal{C}(\one_\mathcal{C})$ defined as the composition 
	\[
	\one_\mathcal{C}\xrightarrow{\delta_x}x \otimes x^\vee \xrightarrow{f\otimes id} x \otimes x^\vee \xrightarrow{\tau}x^\vee \otimes x\xrightarrow{ev_x}\one_\mathcal{C} .
	\]
	In particular, taking $f=id_x$, we have the {\em categorical Euler characteristic} $\chi_\mathcal{C}(x):=tr_x(id_x)$.
\end{Def}

\begin{Def} \label{chicdef}
	Let $k$ be a field {of characteristic 0}, $q: X \rightarrow k$ a $k$-scheme and $\alpha \in \SH(X)$ a constructible, and hence strongly dualisable, object.
	Then $\chic(\alpha /k)$ is defined to be the categorical Euler characteristic of  $q_! \alpha $ in $\SH(k)$:
	\[ \chic(\alpha /k):=\chi_{\SH(k)}(q_! \alpha).
	\]
	In particular we define
	\[ \chic(X/k) := \chic(\one_X/k) =\chi(q_! \one_X):=tr(id_{q_! \one_X})\]
	for every finite type $k$-scheme $q: X \rightarrow \Spec k$. This is well defined as $q_! \one_X$ is strongly dualisable. {We write $\chic(\alpha)$ for $\chic(\alpha/k)$ when the base-field $k$ is clear from the context.} Since $\End_{\SH(k)}(\one_l) \simeq \GW(k)$, the value of the motivic Euler characteristic is an expression in quadratic forms and we also call it the \emph{quadratic} Euler characteristic. For more details see \cite[Section 1]{Le20}, \cite[Section 2]{Az}. 
\end{Def}
The Euler class $e\{X\}$ in this section may be any of the Euler classes constructed in this paper, either the one by \emph{Definition I} (\ref{Eulerclassex1}) or the one by \emph{Definition II} (\ref{Eulerclassex2}), and denote its Borel-Moore value group by $E(X)$.

\begin{Prop}[cut and paste]  Let $X, Y \in \Sch_k^{red}$.\\
	 	(a) $ \chi_{c}(X\amalg Y)=\chi_{c}(X)+\chi_{c}(Y)$. \\
		(b) $e\{X\amalg Y\}=i_X^*e\{X\}+i_Y^*e\{Y\}$ in ${E}_0(X \amalg Y)$ where $i_X$ and $i_Y$ are the closed immersions.
\end{Prop}

\begin{proof}
	For (a) see \cite[Proposition 2.11]{Az}; (b) is in fact the additivity property as in Proposition \ref{classvariety}. 
\end{proof}

\begin{Th}[Singular Gauss-Bonnet] \label{singGB}
	
	Let $ \pi: X\rightarrow \Spec k $ a proper algebraic variety over a field of characteristic zero. Let $E \in \SH(k) $ be a motivic spectrum and $i:\one_{k}\rightarrow E$ the unit map. Then
	
	\[
	i_{*}\chi_{c}(X)=\pi_{*}e^E\{X\} .
	\]
	
\end{Th}

\begin{proof}
	For $X$ smooth and proper the statement is the motivic Gauss-Bonnet theorem of \cite[Theorem 4.6.1]{DJK}, see also  \cite{Le18} on $SL$-oriented theories.
	For a general variety $X$ the theorem follows from the smooth and proper case, the existence of a smooth stratification, and the cut-and-paste formulas above.
\end{proof}

Let now $A$ be an abelian group.
	Given a constructible function $f: X \rightarrow A $ write
	$$ f=\sum{n_i 1_{Z_i}} $$ with $ Z_i $ locally closed subvarieties.
	Define 
	$$ e\{f\}=\sum{n_i e\{ Z_i \} } $$
	and
	$$\chi_{c}(X,f)=\sum{n_i \chic (Z_i)} .$$

\begin{Cor}
	For a constructible function $f:X \rightarrow A$
	\[ i_{*}\chi_{c}(X,f)=\pi_{*}\{f\} .\]
	
\end{Cor}

\appendix
\section{Defining pro-classes}

We first recall some the definitions of normal crossing divisors, resolution of singularities, and weak factorization.

 \label{section:resofsingandweakfac}


\begin{Def} \label{Defncd}
	Let $X$ be in $\Sm_k$ and let $D$ be a reduced effective divisor on $X$ with irreducible components $D_1,\ldots, D_r$. For $I\subset \{1,\ldots, r\}$ we have the corresponding {\em stratum} of $D$, $D_I:=\cap_{i\in I}D_i$; let $|I|$ denote the cardinality of $I$. Recall that $D$ is a {\em simple normal crossing divisor} on $X$ if each $D_I$  is a smooth closed subscheme of $X$ of codimension $|I|$, in other words, the components $\{D_i\mid i\in I\}$ intersect properly and transversely on $X$. Let $W$ be an irreducible smooth closed subscheme of $X$, and let $I\subset \{1,\ldots, r\}$ be the index of maximal cardinality such that $W\subset D_I$. We say that $W$ {\em intersects $D$ with normal crossing} if for each $J\subset\{1,\ldots, r\}\setminus I$, $D_J$ intersects $W$ properly and transversely.  Equivalently, letting $D_W$ be the divisor $\sum_{j\in \{1,\ldots, r\}\setminus I}D_j\cap W$, $D_W$ is a simple normal crossing divisor on $W$. For $W$ an arbitrary smooth closed subscheme of $X$, we say that $W$ intersects $D$ with normal crossing if each irreducible component of $W$ does so.
\end{Def}

\begin{Def} \label{resolutionsing}
	We say that $k$ {\em admits resolution of singularities} if\\
	1. For each $X$ in $\Sch_k^{red}$, there is a proper birational morphism $p:X'\to X$ with $X'$ smooth over $k$, with $p$ a sequence of blow-ups with smooth {centres lying over the singular locus $X_{sing}$ of $X$}.\\
	2. Let $X$ be in $\Sch_k^{red}$ and let $U\subset X$ be a dense open subscheme. Suppose that $U$ is smooth over $k$. Then there exists a proper birational morphism $p:X'\to X$ with $X'$ smooth over $k$ such that $p$ a sequence of blow-ups with smooth centre lying over $X\setminus U$ and such that the complement $X'\setminus p^{-1}(U)$ is a simple normal crossing divisor on $X'$.\\
	3. Let $f:Y\dashrightarrow X$ be a rational map with $X, Y\in \Sch^{red}_k$, and suppose that $f$ is a morphism on a smooth dense open subscheme $U$ of $Y$. Then there is  proper birational morphism $p:Y'\to Y$ such that $p$ a sequence of blow-ups with smooth centres lying over $Y\setminus U$ and such that $f\circ p$ is a morphism. \\
	4. Let $X$ be in $\Sm_k$ and let $D$ be a simple normal crossing divisor on $X$. {Let $W\subset X$ be} a closed subscheme, none of {whose} components is contained in $D$, such that $W \setminus D$ is smooth. Then there exists a proper birational morphism $p: X' \rightarrow X$, with $X'$ smooth over $k$, such that $p$ is a sequence of blow-ups with smooth centre lying over $D\cap W$ and such that the strict transform $W' := p^{-1} [W] = \overline{p^{-1}(W \setminus D)} $ is smooth over $k$. {In addition, letting $D'=p^{-1}(D)_{red}$, $D'$ is a simple normal crossing divisor on $X'$ and $W'$ intersects $D'$ with normal crossing.}
\end{Def}
\begin{Def} \label{weakfact}
	We say that $k$ {\em admits weak factorization} if given $X, Y\in \Sm_k$, both proper over $k$,{$U\subset X$ and $V\subset Y$ dense open subsets}, and a rational map $f:X\dashrightarrow Y$ which {restricts to an isomorphism $f_{VU}:U\to V$},{there is a sequence of  rational maps $p_i:X_i\dashrightarrow X_{i+1}$, $i=0, \ldots, N-1$, such that\\
		a. $X_0=X$, $X_N=Y$ and for each $i$, either $p_i$ or $p_i^{-1}$ is a morphism. If $p_i$ is a morphism, then $p_i$ the morphism given by the blow-up of smooth centre $F_{i+1}\subset X_{i+1}$; if   $p_i^{-1}$  a morphism, then $p_i^{-1}$ is the   morphism given by the blow-up of smooth centre $F_i\subset X_i$.\\
		b. There is an index $r$ such that for each $i\le r$, the induced rational map $g_i: X_{i}\dashrightarrow X$ is a morphism, and for each $i\ge r$,  the induced rational map $f_i: X_i\dashrightarrow Y$ is a morphism. Moreover, for $i\le r$, $g_i$ is an isomorphism over $U$ and for $i\ge r$, $f_i$ is an isomorphism over $V$. \\
		c.  Suppose that $X\setminus U$ and $Y\setminus V$  are simple normal crossing divisors. Then for  $i\le r$, $D_i:=X_i\setminus g_i^{-1}(U)$ is a simple normal crossing divisor on $X_i$ and for $i\ge r$, $D_i:=X_i\setminus f_i^{-1}(V)$ is a simple normal crossing divisor on $X_i$. Moreover, each smooth centre $F_j\subset X_j$ is contained in $D_j$ and  intersects $D_j$ with normal crossing.  	}
\end{Def}

\begin{Rem}
	Every field of characteristic 0 admits resolution of singularities \cite[Main Theorem I]{Hir} and weak factorization \cite[Theorem 0.0.1]{AKMW}.
\end{Rem}
	
\begin{Lemma} \label{classsmooth}
	Suppose that $k$ admits weak factorization, and let $U\in \Sm_k$. Let $E$ be a Borel-Moore homology functor.  In order to define an element 
	$c\in \hat{E}(U)$, it suffices to assign a family of elements $c^i\in E(Z^i)$, one for each good closure $U \hookrightarrow Z^i$, such that for each morphism of {good} closures $\pi : i \to j$, with $\pi: Z^i \to Z^j$ being a blow-up along a smooth centre, {intersecting} $Z_j \setminus j(U)$ with normal crossing,   we have $\pi_*(c^i)=c^j$. 
\end{Lemma}

\begin{proof}By definition of $\hat{E}(U)$, an class $c \in \hat{E}(X)$ is defined by a class $c^i \in E(Z^i)$ for each good closure $X \in Z^i$, satisfying for each morphism of  that $f_*(c^j)=c^i$. It remains to see that we can check this compatibility condition only on morphisms $\pi$ which are blow-ups as in the lemma.
	
	Consider the following situation: $U \to Z^i$, $i=1,2,3$ are good closures, with proper morphisms $\alpha : (U, Z_2) \to (U,Z_1)$, $\beta : (U, Z_2) \to (U, Z_3)$ and $\gamma : (U,Z_1) \to (U, Z_3)$, satisfying $ \beta = \gamma \circ \alpha $. Assume we have $c^i \in E(Z_i)$ for each $i$. Then if $\alpha_* c^2 = c^1 $ and $\beta_* c^2 = c^{3}$, we evidently have $\gamma_* c^1 = \gamma_* \alpha_* c^2 = \beta_* c^2 = c^3 $.
	
	{Now let $f$ be a morphism of  good  closures $f : (i:U\hookrightarrow Z) \to (j: U\hookrightarrow W)$. Then $f: Z \to W $ is a proper birational map and $f\circ i=j$. By weak factorization, retaining notation from \ref{weakfact}, we can decompose $f$ as a composition of birational maps $p_{N-1}\circ\ldots\circ p_1\circ p_0$ such that either  $p_i: X_i \to X_{i+1}$ is the blow-up of $X_{i+1}$ along a smooth centre $F_{i+1}\subset X_{i+1}$ or $p_i^{-1}:X_{i+1}\to X_i$ is the blow-up of $X_i$ along a smooth centre $F_i\subset X_i$; here $Z=X_0$ and $W=X_N$.  Moreover, the induced rational map $U\dashrightarrow X_i$ is a morphism, defining a good closure $\alpha_i:U\to X_i$; let $D_i$ be the simple normal crossing divisor $X_i\setminus \alpha_i(U)$. Finally, if $p_i$ is a morphism, then $F_{i+1}\subset D_{i+1}$ and intersects $D_{i+1}$ with normal crossing, and if $p_i^{-1}$ is a morphism, then $F_i\subset D_i$  and intersects $D_i$ with normal crossing.   }

	Each $X_i$ is a good closure of $U$, so we have a well-defined class $c^i \in E(X_i)$.  Let $r$, $0 \leq r \leq N $ be as in \ref{weakfact}.
	{We have the induced morphisms $g_i:X_i\to Z$ for  $i\le r$ and $f_i:X_i\to W$ for $i\ge r$. For $i\le r$, $g_i$ defines a morphism of good closures $g_i:\alpha_i\to \alpha_0$ and for $i\ge r$,  $f_i$ defines a morphism of good closures $f_i:\alpha_i\to \alpha_N$.  Take $i\le r-1$ and suppose that $p_i:X_i\to X_{i+1}$ is a morphism. Then we have the morphism of good closures $p_i:\alpha_i\to \alpha_{i+1}$ which is the blow-up along a smooth centre intersecting $X_{i+1}\setminus \alpha_{i+1}(U)$ with normal crossing, so by assumption $p_{i*}(c^i)=c^{i+1}$. Thus 
		\[
		g_{i+1*}(c^{i+1})=g_{i+1*}\circ p_{i*}(c^i)=g_{i*}(c^i)
		\]
		If $i\le r$ and $p_i^{-1}$ is a morphism, we similarly have $g_{i*}(c^{i})=g_{i-1*}(c^{i-1})$; by induction, we thus have $g_{r*}(c^r)=c^0=c_i$. }
	
	Similarly, if $i\ge r$, we have $f_{i*}(c^i)=c^N=c_j$. Our original map $f$ is a morphism and we have the factorization  $f\circ g_r=f_r$, thus
		\[
		f_*(c^i)=f_*(g_{r*}(c^r))= (f\circ g_r)_*(c^r)=f_{r*}(c^r)=c_j.
		\]
\end{proof}

\begin{Def} Take $X\in \Sch^{red}_k$. \\
	 A {\em stratification} of $X$ is a choice of finite family of locally closed subschemes $\{U_i\}_{i\in I}$ {of $X$} such that $X$ is the disjoint union of the $U_i$. A stratification is called a {\em smooth} stratification if each of the $U_i$ is smooth over $k$.\\
	 A {\em refinement} of a stratification  $\{U_i\}_{i\in I}$ is a stratification $\{V_j\}_{j\in J}$ such that each $V_j$ is contained in some $U_i$.
\end{Def}

\begin{Rem} If $k$ is a perfect field, then each $X\in  \Sch^{red}_k$ admits a smooth stratification.
	Indeed, the assumption that $k$ is perfect implies that the smooth locus of $X$ is a dense open subscheme $U$, and the existence of a smooth stratification of $X$ containing $U$ as member follows by noetherian induction.
	
		If $k$ admits resolution of singularities or weak factorization, then $k$ is automatically perfect;
We assume that our base-field $k$ is perfect.
\end{Rem}
{Take $U\in \Sm_k$. Suppose we have defined  a system of classes $\{c^i \in E(Z^i)\}_i$ for each good closure $i:U \to Z^i$, compatible with respect to blow-ups along smooth centres as in the statement of Lemma~\ref{classsmooth}. The family $\{c^i \in E(Z^i)\}_i$ thus gives us a well-defined class $c\{U\}\in \hat{E}(U)$. Assuming that the classes $c\{U\}$ satisfy an additivity property with respect to a decomposition into an open subscheme and closed complement, we can extend our classes to any variety $X\in \Sch^{red}_k$ by choosing a stratification $\sU:=\{U_i\}_{i\in I}$ of $X$, and defining $c\{X\} =c_\sU:=\sum_{i\in I}\alpha_{i*}(c\{U_i\})$, where $\alpha_i:U_i\to X$ is the inclusion. The main point is that under this additivity assumption, the resulting class in $\hat{E}(X/k)$ is independent of the choice of smooth stratification of $X$.}

\begin{Prop} \label{classvariety}
	{Suppose that for each $U \in \Sm_k $ we have a   class $c\{U\} \in \hat{E} (U/k)$. Suppose in addition that for each smooth $U$ with a smooth closed subscheme $i: Z\hookrightarrow U$, and with open complement $j:U\setminus Z\hookrightarrow U$, we have}
	\begin{equation}\label{PropHypoth}
		c\{U\} = i_*c\{Z\} +j_* c\{U \setminus Z\} . 
	\end{equation}
	
	{Then for $X \in \Sch_k^{red}$, the element $c_\sU\{X\}\in \hat{E}(X)$ is independent of the choice of smooth stratification $\sU$ of $X$; we denote $c_\sU\{X\}$ by $c\{X\}$. Moreover, if we have an open subscheme $:U\to X$ with closed complement $i:Z\to X$, then $c\{X\}=i_*(c\{Z\})+j_*(c\{U\})$. }
\end{Prop}

\begin{proof}
	We use here again Bittner's theorem \cite[Theorem 5.1]{Bitt} on the presentation of the relative Grothendieck group, this time instead of the presentation by blow-up relations, we use the presentation by cut-and-paste relations, namely that the  abelian group $K_0(\Var_{/S})$ is generated by $S$-varieties $X$ which are smooth over $k$, with relations of additivity with regards to any smooth subvariety $Y \subset X$ and its complement $X \setminus Y$. With view of this presentation, condition \ref{PropHypoth} just means that the mapping
	\begin{align*}
		{\Sm_k}_{/X} &\to \hat{E} (X/k) \\
		(Z \xrightarrow{f} X) &\mapsto f_*c\{Z\} 
	\end{align*}
	factors through $K_0(\Var_{/X})$ for every $X \in \Sch^{red}_k$, giving a map $c_X : K_0(\Var_{/X}) \to \hat{E} (X/k)$ . As a consequence we get a well defined class $c\{X\}:= c_X \{[id: X \to X]\}$. It is clear that for a stratification $\sU$, $c\{X\} = c_{\sU}\{X\}$ and this does not depend on the stratification.
\end{proof}

We now formulate a condition on a family $\{c_U^{\bar{U}}\in E(\bar{U}/k)\}_{U\in \Sm_k, \bar{U}\in \mathcal{C}^g(U)}$, that will ensure we can define $c\{X\}$ for every $X\in \Sch_k^{red}$. This follows the same approach as in \cite{Alu}.

\begin{Def}[good local data] \label{goodlocaldata}
	Suppose we  are given, for each $U\in \Sm_k$ and each good compactification $i:U\to\bar{U}$ an element $c_U^{\bar{U}}\in E(\bar{U}/k)$; if $U$ happens to be proper over $k$, we write $c(U)$ for $c(U,id_U)$.{We suppose in addition that if $\bar{U}$ is a disjoint union, $\bar{U}=\bar{U}_1\amalg \bar{U}_2$, with inclusions $j_i:\bar{U}_i\to \bar{U}$, then letting $U_i=U\cap\bar{U}_i$, we have
		\begin{equation}\label{additiivity}
			c_U^{\bar{U}}=j_{1*}(c_{U_1}^{\bar{U}_1})+j_{2*}(c_{U_2}^{\bar{U}_2})
	\end{equation}}

	{Let $(U,\bar{U}, W )$ be a triple with
		\begin{itemize}[noitemsep]
			\item   $\bar{U}\in \Sm_k$ a smooth   variety,
			\item $U$ a dense open subscheme of $\bar{U}$,
			\item $W$ a smooth closed subvariety of $\bar{U}$,
		\end{itemize}
		such that  $D:= \bar{U} \setminus U$ is a simple normal crossing divisor on $\bar{U}$ and $W$ intersects $D$ with normal crossing. }
	
	{Let $\pi: {\bar{V}} = Bl_W(\bar{U}) \to \bar{U}$ be the blow-up of $\bar{U}$ along $W$, with exceptional divisor $F$, and let $w:W\to \bar{U}$ be the inclusion. Let $\tilde{D}$ be the proper transform $\pi^{-1}[D]$.   Since $W$ intersects $D$ with normal crossing,  $\tilde{D}+F$ is a  is a simple normal crossing divisor on $\bar{V}$. Let $Z:=W \cap U$, let $V:=\pi^{-1}(U)$,  let $E:=  \pi^{-1}(Z)=F \cap V$. Note that  $\bar{U}$ is a good compactification of $U$, $\bar{V}$ is a good closure of $V\setminus E$, so $c_U^{\bar{U}}$ and $c_{V \setminus E}^{\bar{V}}$ are defined.  Letting $W_Z\subset W$ be the closure of $Z$ in $W$, $W_Z$ is the union of the irreducible components of $W$ that are not contained in $D$, and $Z\subset W_Z$ is a good closure of $Z$. We write $c_Z^W$ for  the image of $c_Z^{W_Z}$ in $E_n(W/k)$ via the pushforward for the inclusion $W_Z\hookrightarrow W$; if $Z=\emptyset$, we set $c_Z^W=0$.}
	
	{We say  the family $\{c_U^{\bar{U}}\in E(\bar{U}/k)\}_{U\in \Sm_k, i\in \mathcal{C}^g(U)}$  {\em  forms good local data} if  for each triple $(U,\bar{U}, W)$ as above,  we have
		\begin{equation}\label{gld}
			c_U^{\bar{U}} = \pi_* c_{V \setminus E}^{\bar{V}} + w_* c_Z^W\in {E}(\bar{U}).
	\end{equation}}
	
\end{Def}

{\begin{Rem}\label{ReducibleGLD}  One could also frame Definition~\ref{goodlocaldata} with the additional requirement that $W$ and $U$ are irreducible. In fact, given a $\{c_U^{\bar{U}}\in E(\bar{U}/k)\}_{U\in \Sm_k, i\in \mathcal{C}^g(U)}$ that satisfies this restricted set of conditions, it follows from the assumed additivity of the classes $c_U^{\bar{U}}$ with respect to disjoint union, that this family forms good local data in the sense of the full Definition~\ref{goodlocaldata}. It is however easier to apply Definition~\ref{goodlocaldata} to the problem at hand if we include the case of reducible $U$ and $W$ in the definition.
\end{Rem}}

{
	\begin{Lemma}\label{TripleCover} Take $U\in \Sm_k$ and let $i:Z\to U$ be a closed immersion with $Z\in \Sm_k$. Let $i_U:U\to U'$ be a good closure. Then  there is a triple $(U, \bar{U}, W)$ satisfying the conditions of Definition~\ref{goodlocaldata} with $Z=W\cap U$ dense in $W$, such that the rational map $\bar{U}\dashrightarrow U'$  induced by the identity map on $U$ is a morphism, and defines a map of good closures 
		$(U\hookrightarrow \bar{U})\to (U\hookrightarrow U')$.
\end{Lemma}}

{\begin{proof} Let $W''\subset U'$ be the closure of $i_U(Z)$ and let $D'=U'\setminus U$. By  resolution of singularities, Definition \ref{resolutionsing}(4), there is a proper birational morphism $\pi:\bar{U}\to U'$ that is a sequence of blow-ups along  smooth centres lying over $W'\cap D'$, such that $D:=\pi^{-1}(D')_{red}$ is a simple normal crossing divisor on $\bar{U}$, the proper transform $W:=\pi^{-1}[W']$ is smooth, and $W$ intersects $D$ with normal crossing. In other words, the triple $(U, \bar{U}, W)$ satisfies the conditions of Definition~\ref{goodlocaldata}, and $Z=W\cap U$ is dense in $W$. By construction the rational map $\bar{U}\dashrightarrow U'$  induced by the identity map on $U$ is a morphism, and thus define maps of good compactifications 
		$(U\hookrightarrow \bar{U})\to U\hookrightarrow U')$.
\end{proof}}

\begin{Prop} \label{goodclass}
	Suppose that $k$ admits resolution of singularities and weak factorization. Suppose we are given for each $U\in \Sm_k$ and each good closure  $i:U\to \bar{U}$,  an element $c_U^{\bar{U}} \in E(\bar{U}/k)$ such that the family $\{c_U^{\bar{U}}\in E(\bar{U}/k)\}_{U\in \Sm_k, \bar{U}\in \mathcal{C}^g(U)}$  forms good local data. Then\\
	1. For each $U\in \Sm_k$, the family of elements $\{c_U^{\bar{U}}\}_{\bar{U}\in \mathcal{C}^g(U)}$ defines an element $c\{U\} \in \hat{E}(U/k)$.\\
	2. For  $X$ in $\Sch^{red}_k$ choose a smooth stratification $\sU:=\{U_i\}_{i\in I}$  with inclusions $\alpha_{U_i}:U_i\to X$. Then $c_\sU(X) :=\sum_{i\in I}\alpha_{i*}(c\{U_i\})) \in \hat{E}(X/k)$ is independent of the choice of $\sU$, giving rise to the well-defined element
	\[
	c\{X\}: = c_\sU(X)\in \hat{E}(X/k).
	\]
\end{Prop}

\begin{proof}

	{For (1), we show that  the family  $\{c_U^{\bar{U}}\}_{i\in \mathcal{C}^g(U)}$ satisfies the hypothesis of Lemma \ref{classsmooth}.  Given a map of good closures $f:(i:U\to Z^i)\to (j:U\to Z^j)$ such that $f$ is the blow-up of $Z^j$ along a smooth centre $W$ which intersects $Z^j\setminus j(U)$ with normal crossing, we take $\bar{U}=Z^j$. This gives us the triple $(U, \bar{U},W)$ satisfying the conditions of  Definition~\ref{goodlocaldata}, which has $Z=\emptyset$ and for which $\bar{V}:=Bl_W\bar{U}\to \bar{U}$ is the map $f:Z^i\to Z^j$.  The identity \eqref{gld}  gives the  identity
		\[
		f_*(c_U^{Z^i})=c_U^{Z^j}.
		\]
		Having verified the hypothesis of Lemma \ref{classsmooth}, (1) follows.}

	{We now prove (2). It follows from Lemma~\ref{TripleCover} that the good closures of $U$   of the form $U\subset \bar{U}$,   with $(U, \bar{U}, W)$ a triple satisfying the conditions of Definition~\ref{goodlocaldata}, such that $Z=W\cap U$ and with $Z$ dense in $W$, gives a cofinal subcategory of $\sC^g(U)$. Thus, the collection of  identities \eqref{gld} for each such triple $(U, \bar{U}, W)$ suffices to show that 
		$c\{U\} = i_*c\{Z\} +j_* c\{ U \setminus Z\}$ in $\hat{E}(U/k)$.  By Proposition~\ref{classvariety} this shows that the classes $c_\sU\{X\}\in \hat{E}(X/k)$ for $\sU$ a smooth stratification of $X\in \Sch_k^{red}$ are independent of the choice of $\sU$, proving (2). }
\end{proof}

\small
\bibliographystyle{alphamod}

\let\mathbb=\mathbf

\bibliographystyle{plain}

\begin{flushleft}
	\small \textsc{Institut f\"ur Mathematik, Universit\"at Zürich \\
		Winterthurerstrasse 190, 8057 Z\"urich, Switzerland} \\
	ran.azouri@math.uzh.ch
\end{flushleft}

\end{document}